\documentclass[12pt]{amsart}
\usepackage{graphicx}
\newtheorem{theorem}{Theorem}[section]
\newtheorem{proposition}[theorem]{Proposition}
\newtheorem{lemma}[theorem]{Lemma}

\newtheorem{corollary}[theorem]{Corollary}

\theoremstyle{remark}
\newtheorem{remark}[theorem]{Remark}

\newtheorem{notation}[theorem]{Notation}

\newtheorem{definition}[theorem]{Definition}

\numberwithin{equation}{section}

\begin{document}

\title[Non-orientable fundamental surfaces in lens spaces] {
Non-orientable fundamental surfaces in lens spaces}

\author{Miwa Iwakura}

\date{April 20, 2009}


\begin{abstract}
 We consider non-orientable closed surfaces of minimum crosscap number 
in the $(p,q)$-lens space $L(p,q) \cong V_1 \cup_{\partial} V_2$,
where $V_1$ and $V_2$ are solid tori.
 Bredon and Wood gave a formula for calculating the minimum crosscap number.
 Rubinstein showed that
$L(p,q)$ with $p$ even has only one isotopy class of such surfaces,
and it is represented by a surface in a standard form,
which is constructed from a meridian disk in $V_1$
by performing a finite number of band sum operations in $V_1$
and capping off the resulting boundary circle by a meridian disk of $V_2$.
 We show that the standard form corresponds 
to an edge-path $\lambda$ 
in a certain tree graph in the closure of the hyperbolic upper half plane.
 Let $0=p_0/q_0,\ p_1/q_1,\ \cdots,\ p_k/q_k = p/q$ be the labels of vertices
which $\lambda$ passes.
 Then the slope of the boundary circle
of the surface right after the $i$-th band sum is $(p_i, q_i)$.
 The number of edges of $\lambda$ is equal to the minimum crosscap number.
 We give an easy way of calculating $p_i / q_i$
using a certain continued fraction expansion of $p/q$.
\end{abstract}

\maketitle

\section{Introduction}

 A solid torus $V$ is homeomorphic to $D^2 \times S^1$,
where $D^2$ is the $2$-dimensional disk and $S^1$ the $1$-dimensional sphere.
 A circle on the boundary torus $\partial V \cong (\partial D^2) \times S^1$ 
is of {\it $(p,q)$-slope} (or $p/q$-slope)
if it is isotopic to the circle given by the expression
\newline
$( (\cos 2\pi q \theta, \sin 2\pi q \theta), 
   (\cos 2\pi p \theta, \sin 2\pi p \theta) )$
where $p$ and $q$ are coprime integers
and $\theta$ is a parameter with $0 \le \theta \le 1$.
 We call a circle of $(1,0)$-slope a {\it longitude}
and that of $(0,1)$-slope a {\it meridian}.
 In general, a circle in a surface is said to be {\it essential}
if it does not bound a disk in the surface.
 As is well-known,
in the boundary torus $\partial V$,
any essential circle is of $(p,q)$-slope for some coprime integers $p, q$.

 For a pair of coprime integers $p$ and $q$ with $p \ge 2$,
the {\it $(p,q)$-lens space} $L(p,q)$ 
is obtained from two solid tori $V_1$ and $V_2$
by gluing their boundary tori by a homeomorphism 
which maps the meridian circle on $\partial V_2$ 
to a circle of $(p,q)$-slope on $\partial V_1$.
 Throughout this paper,
we regard $L(p,q)$ as the union of $V_1$ and $V_2$ as above.
 Since $L(p,q) \cong L(p,-q)$ and $L(p,q) \cong L(p, q+p)$,
it is enough for establishing a general result for $(p,q)$-lens spaces
to consider $L(p,q)$ with $1 \le q \le 2/p$.

 It is well-known that $L(p,q)$ contains a non-orientable closed surface
if and only if $p$ is even.
 This holds because
$H_2 (L(p,q); {\Bbb Z}_2) = {\Bbb Z}_2$ when $p$ is even,
$H_2 (L(p,q); {\Bbb Z}_2) = 0$ when $p$ is odd,
and $H_2 (L(p,q);{\Bbb Z}) = 0$ for any integer $p \ge 2$.
 In \cite{BW}, Bredon and Wood gave a formula 
for calculating the minimum crosscap number ${\rm Cr}(p,q)$ 
among those of all non-orientable connected closed surfaces in $L(p,q)$.
 The formula is based on a continued fraction expansion of $p/q$.

\begin{notation}
 For a finite sequence of real numbers $a_0, a_1, \cdots, a_n$, 
we let 
\newline
$[a_0,a_1,\cdots,a_{n-1},a_n]$ 
denote 
$a_0+\dfrac{1}{
\begin{array}{c}
a_1 + \\
 \\
 \\
\end{array}
\begin{array}{c}
\ddots
 \\
 \\
\end{array}
\begin{array}{c} 
\\
+\dfrac{1}{a_{n-1} + \dfrac{1}{a_n}} \\
\end{array}
}
\in {\mathbb R} \cup \{ \infty \}$,
where $r/0 = \infty$, $r / \infty = 0$
and $\infty + r = \infty$ for any real number $r$,
\end{notation}

\begin{definition}
 Let $r$ be a positive rational number.
 A continued fraction expansion
$r = [a_0, a_1, \cdots, a_n]$ is said to be {\it standard}
if $a_0$ is a non-negative integer, 
$a_i$ is a natural number for $i=1, 2, \cdots, n$
and $a_n \ge 2$.
 By considering Euclidiean method of mutual division,
it is easily seen that such an expression is unique for $r$.
\end{definition}

\begin{theorem} {\rm (Bredon and Wood, (6.1) Theorem in \cite{BW})}
 Let $p,q$ be coprime natural numbers with $p \ge 2$,
and $p/q = [a_0, a_1, \cdots, a_n]$
the standard continued fraction expansion.

 Then the minimum crosscap number is calculated by 
${\rm Cr}(p,q)= \sum_{i=0}^n b_i/2$,
\newline
where 
$b_0 = a_0$, and 
$b_i = 
\left\{ \begin{array}{l}
a_i 
\ ({\rm when}\ b_{i-1} \ne a_{i-1} \ 
{\rm or} \ \sum_{j=0}^{i-1} b_j \ {\rm is \ odd})\\ 
0 \ \ ({\rm otherwise}) \\
\end{array} \right.$
\end{theorem}

 J. H. Rubinstein studied 
non-orientable closed surfaces in $3$-manifolds 
in \cite{R1}, \cite{R2} and \cite{R3}.
 Such surfaces of minimum crosscap number in lens spaces are considered
in section 3 in \cite{R3}.
 See also \cite{F}, 
in which one-sided closed surfaces in Seifert fibered spaces are studied
by C. Frohman.
 In order to introduce results in \cite{R3},
we need to recall some definitions.

\begin{definition}
 Let $M$ be a closed $3$-manifold,
and $F$ a (possibly one-sided) closed surface embedded in $M$.
 An embedded disk $D$ in $M$ is called a compressing disk of $F$
if $D \cap F = \partial D$ 
and the boundary circle $\partial D$ is essential in $F$.
 We say $F$ is {\it geometrically incompressible}
if it has no compressing disk.
 If it has, it is {\it geometrically compressible}.
\end{definition}

\begin{remark}\label{remark:NonOriIncpr}
 A non-orientable closed surface 
of minimum crosscap number in $L(p,q)$ is 
geometrically incompressible
as shown in lines 9--11 in page 192 in \cite{R3}.
 Section \ref{section:GeometricalIncompressibility}
includes the argument for self-containedness.
\end{remark}

\begin{figure}
\includegraphics[width=8cm]{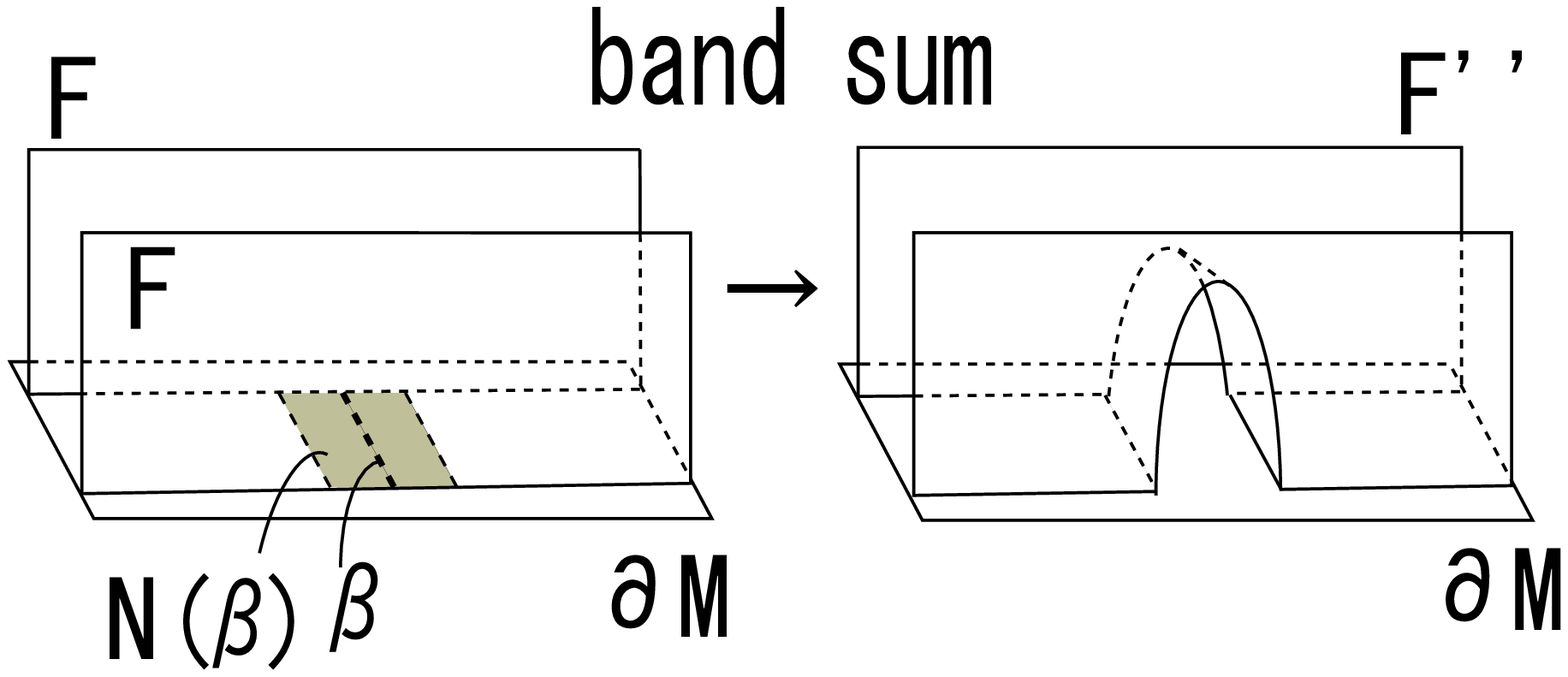}
\caption{}
\label{fig:BandSum}
\end{figure}

\begin{definition}
 Let $M$ be a compact $3$-manifold with non-empty boundary $\partial M$,
$F$ a compact surface properly embedded in $M$,
and $\beta$ an arc embedded in $\partial M$ 
so that $\beta \cap \partial F = \partial \beta$.
 We obtain a new surface $F''$ 
by an operation called a {\it band sum} on $F$ along $\beta$ as below.
 We take a tubular neighbourhood $N(\beta) \cong \beta \times I$ 
of $\beta$ in $\partial M$
so that $\partial F \cap N(\beta) = (\partial \beta) \times I$.
 We isotope the interior of the surface $F' = F \cup N(\beta)$ 
slightly into int\,$M$
with its boundary circles $\partial F'$ fixed, 
to obtain a new surface $F''$.
 See Figure \ref{fig:BandSum}.
 A band sum is {\it trivial}
if a union of $\beta$ and a subarc of $\partial F$
forms a circle bounding a disk in $\partial M$.
\end{definition}

\begin{remark}\label{remark:EssBandSum}
 In the above definition of band sum,
if $\partial M$ is a torus and
$\partial F$ is a single essential circle in $\partial M$,
then 
a non-tirivial band sum 
is along an arc connecting the both sides of $\partial F$,
that is, for an arbitrary tubular neighborhood 
$N(\partial F) \cong \partial F \times I$ of $\partial F$ in $\partial M$, 
both $\beta \cap (\partial F \times \{ 0 \}) \ne \emptyset$ 
and $\beta \cap (\partial F \times \{ 1 \}) \ne \emptyset$ hold.
 In this case, the boundary of the resulting surface 
is again a single essential circle in $\partial M$.
\end{remark}

\begin{definition}\label{definition:standard}
 Let $F$ be a connected closed surface embedded 
in $L(p,q) \cong V_1 \cup_{\partial} V_2$.
 Then we say that $F$ is in {\it standard form}
if $F$ is obtained from a meridian disk $D_1$ of $V_1$
by performing a finite number of non-trivial band sum operations
and capping off the boundary circle of the resulting surface
by a meridian disk $D_2$ of $V_2$.
\end{definition}

 Note that $F$ in standard form is non-orientable,
since it intersects a core loop of $V_1$ (resp. $V_2$)
in a single point in the interior of $D_1$ (resp. $D_2$).

\begin{theorem}\label{theorem:Rubinstein}
{\rm (Rubinstein, Proposition 11 and Theorem 12 in \cite{R3})}
 Let $L(p,q) \cong V_1 \cup_{\partial} V_2$ 
be the $(p,q)$-lens space with $p$ even.
 Let $F$ be a geometrically incompressible connected closed surface in $L(p,q)$
which is not homeomorphic to the $2$-sphere.
 Then $F$ can be isotoped to be in standard form.
 Moreover, in $L(p,q)$,
a geometrically incompressible non-spherical connected closed surface 
is unique up to isotopy,
and hence 
is a non-orientable closed surface of minimum crosscap number.
\end{theorem}

 As an example, 
the sequence of band sums for $(8,3)$-lens space 
is described in Figure \ref{fig:83-lens}.

\begin{figure}
\includegraphics[width=8cm]{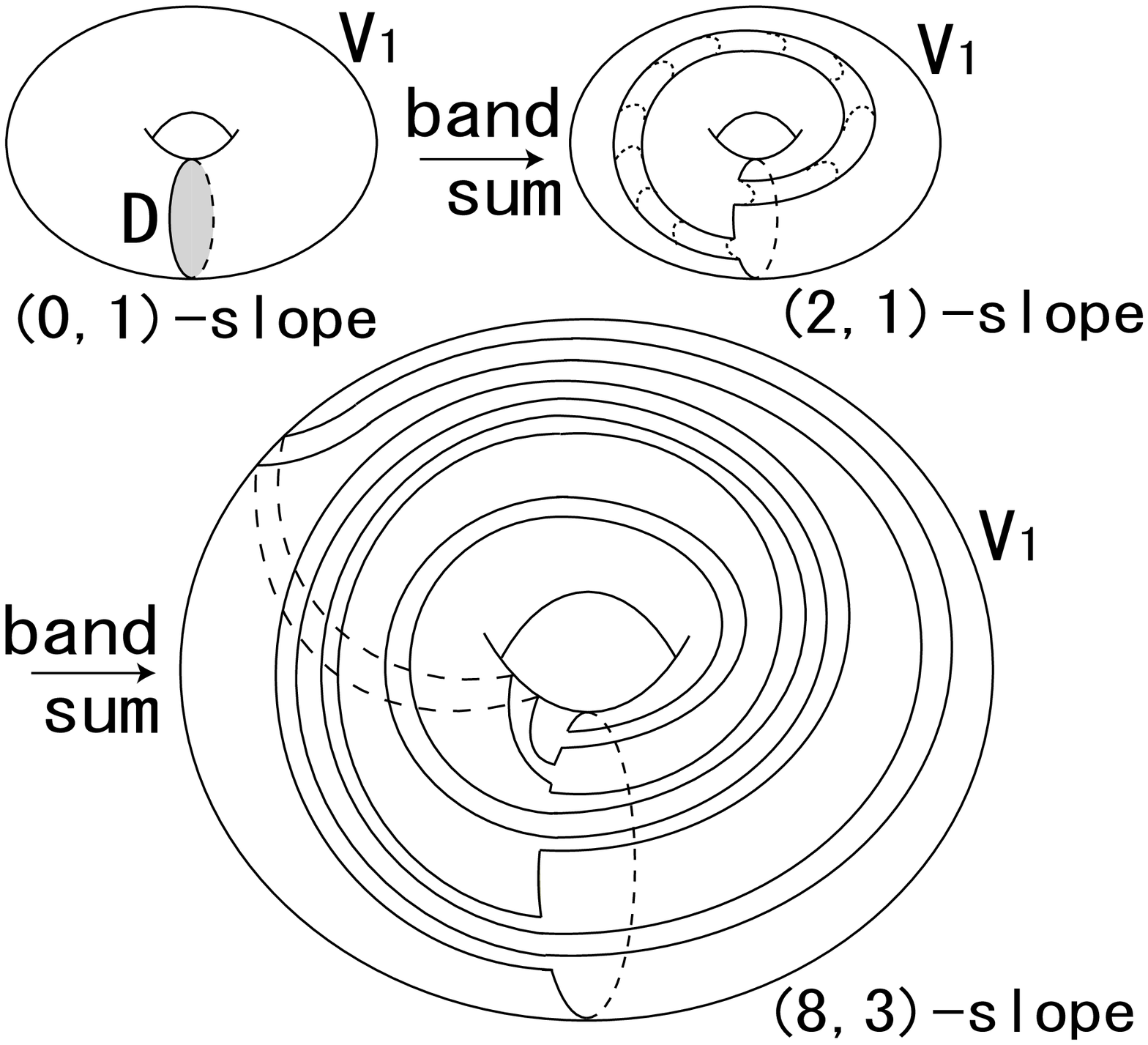}
\caption{}
\label{fig:83-lens}
\end{figure}

 Applications of uniqueness of an isotopy class 
of some kind of non-orientable surfaces
are found in \cite{R1}, \cite{R2}, \cite{R3}, \cite{R4} and \cite{JRT}.

 The next theorem gives
a new formula of the minimal crosscap number
of non-orientable connected closed surfaces in $L(p,q)$,
which is obtained from consideration on standard positions. 

\begin{theorem}\label{theorem:NewFormula}
 Let $p,q$ be coprime natural numbers with $p \ge 2$,
and 
\newline
$p/q = [\alpha_n, \alpha_{n-1}, \cdots, \alpha_0]$ 
be the standard continued fraction expansion.
\newline
 We define 
$\alpha'_0=\alpha_0$, and
$\alpha_i' = 
\left\{ \begin{array}{l}
\alpha_i 
\ ({{\rm when}  \ \alpha'_{i-1}=\infty})\\ 
\alpha_i+1 \ ({{\rm when}  \ \alpha'_{i-1} \  {\rm is \  odd}}) \\
\infty \ ({{\rm when}  \ \alpha'_{i-1} \  {\rm is \  even }}) \\
\end{array} \right.$ 
for $i=1, 2, \cdots, n$.
\newline
 We set 
$\beta_i = 
\left\{ \begin{array}{l}
\alpha'_i / 2 
\ ({{\rm when}  \ \alpha'_i \  {\rm is \  even}})\\ 
(\alpha'_i-1) / 2 \ ({{\rm when}  \ \alpha'_i \  {\rm is  \ odd}}) \\
0 \ ({{\rm when}  \ \alpha'_i=\infty}) \\
\end{array} \right.$
for $i=0, 1, \cdots, n$.
\newline
 Then the minimum crosscap number is calculated by
${\rm Cr}(p,q) = \sum_{i=0}^n \beta_i$.
\end{theorem}

%

\begin{figure}
\includegraphics[width=8cm]{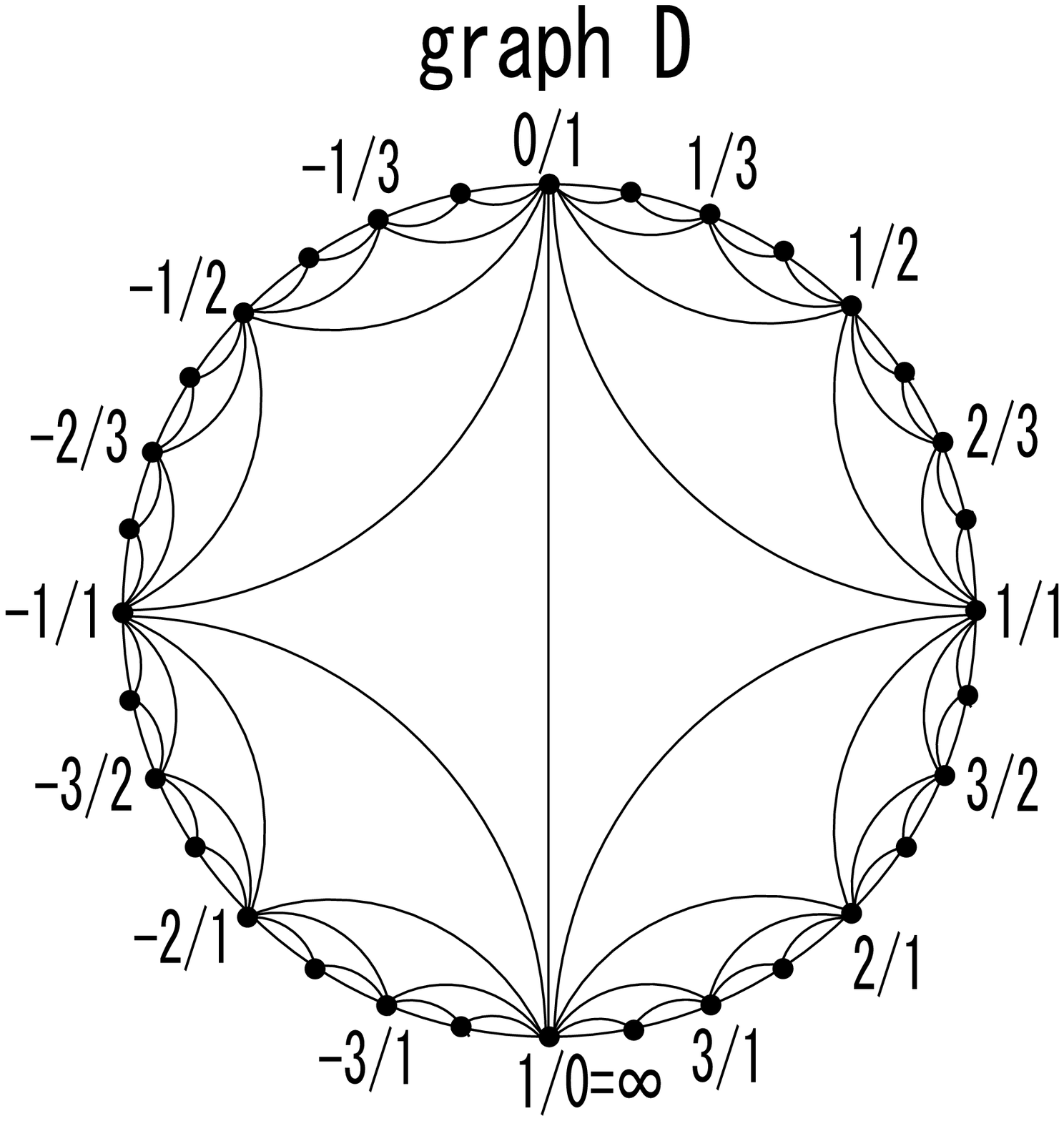}
\caption{}
\label{fig:D}
\end{figure}

 In \cite{HT}, 
transitions of slopes (of circles and arcs in a $2$-sphere with four punctures) 
caused by band sum operations 
is described by edge-paths in the graph ${\mathbb D}$ below.
 In this paper, 
we introduce a certain tree graph ${\mathbb D}_2$ taking after ${\mathbb D}$.

 For a pair of integers $p$ and $q$,
we say $p/q$ is an {\it irreducible fractional number}
if $p$ and $q$ are coprime, that is, ${\rm GCD}(p,q) =1$.
 Hence $p = p/1$ is an irreducible fractional number for any integer $p$.
 We consider $1/0$ and $(-1)/0$ 
representing the same irreducible fractional number,
which is denoted by $\infty$.
 As usual, $\infty + p/q = \infty$ and $1/\infty = 0$.
 For an arbitrary pair of irreducible fractional numbers $p/q$ and $r/s$,
we set 
$d(p/q, r/s) 
= |{\rm det} \left( \begin{array}{cc} p & r \\ q & s \end{array} \right)| 
= |ps-rq|$, 
and call it the {\it distance} of them.

 The graph ${\mathbb D}$ is embedded on the upper half plane 
${\mathbb H} \ (\subset {\mathbb C})$
with the real line ${\mathbb R}$ and the point at infinity $\infty$.
 Its vertices are the rational points and $\infty$
in ${\mathbb R} \cup \{ \infty \}$,
and its edges are geodesics on the upper half model of the hyperbolic plane
which connect two vertices corresponding to the irreducible fractional numbers
$a/c$, $b/d$, ($a,b,c,d \in {\mathbb Z}$)
if and only if $|ad-bc| = d(a/c, b/d) = 1$.
 See Figure \ref{fig:D},
where ${\mathbb H} \cup {\mathbb R} \cup \{ \infty \}$ 
is transformed onto the Poincar\'{e} disk model 
by the map $z \mapsto \dfrac{z+i}{iz+1}$.
 We can easily draw this graph
by following the rule that
two vertices corresponding to the irreducible fractional number $p/q$, $r/s$ 
and connected by an edge of ${\mathbb D}$
are those of a triangle face of ${\mathbb D}$
with $(p+q)/(r+s)$ being the third vertex.

\begin{figure}
\includegraphics[width=8cm]{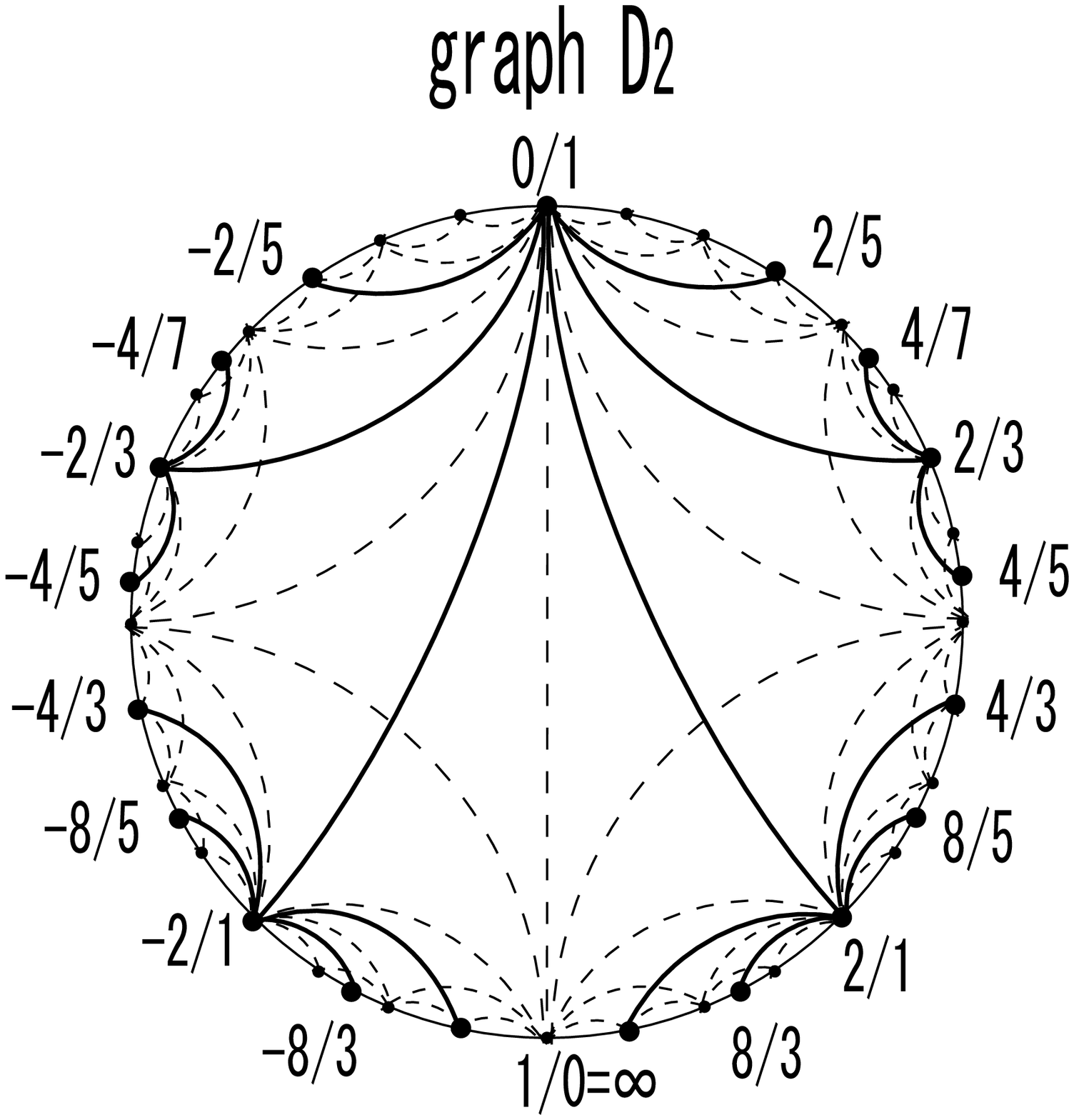}
\caption{}
\label{fig:D2}
\end{figure}

 When $p$ is even,
a vertex of ${\mathbb D}$ 
corresponding to the irreducible fractional number $p/q$
is called an {\it even} vertex in this paper.
 If $p$ is odd, we call it an {\it odd} vertex.
 Adequately separating the trigonal faces of ${\mathbb D}$ 
into adjacent pairs
and taking a union of every pair of trigonal faces,
we obtain a tiling of the upper half plane 
by infinitely many quadrilaterals
with two even vertices and two odd vertices.
 The vertices of ${\mathbb D}_2$ are the even vertices of ${\mathbb D}$.
 Two vertices $p/q$, $r/s$, 
($p,q,r,s \in {\mathbb Z}$, ${\rm GCD}(p,q)=1$ and ${\rm GCD}(r,s)=1$)
are connected by an edge of ${\mathbb D}_2$
if and only if $|ps-rq|=d(p/q,r/s)=2$.
 Each edge of ${\mathbb D}_2$ does not appear in ${\mathbb D}$,
but forms a diagonal line of a quadrilateral of the tiling as above.
 See Figure \ref{fig:D2},
where ${\mathbb D}_2$ is described by solid lines.
 We regard the vertex, 
assigned an irreducible fractional number $r/s$, 
as corresponding to the $(r,s)$-slope on $\partial V_1$.

 The next theorem will be shown in Section \ref{section:SurgeryAndSlope}.

\begin{theorem}\label{theorem:BandSumD2}
 For any band sum operation in Definition \ref{definition:standard},
the slopes of the boundary circles of the surfaces in $V_1$
before and after the operation
are connected by an edge of ${\mathbb D}_2$.
\end{theorem}

 This theorem implies that
the transition of the slopes of the boundaries of the surfaces in $V_1$ 
by the band sums in Definition \ref{definition:standard} is 
along an edge-path of ${\mathbb D}_2$,
in which the same edge can appear twice. 

 We will show the next theorem in Section \ref{section:D2}.

\begin{theorem}\label{theorem:tree}
The graph ${\mathbb D}_2$ is a tree, i.e., 
${\mathbb D}_2$ is connected and contains no cycle.
\end{theorem}

 This theorem also implies the uniquness in Theorem \ref{theorem:Rubinstein},
as we give a proof in Section \ref{section:MainTheorem}.
 However, Rubinstein's proof (of Theorem 12 in \cite{R3})
is very short and clear.

\begin{theorem}\label{theorem:main}
 Let $L(p,q)$, $V_1$, $V_2$, $F$, $F'$ be 
as in Theorem \ref{theorem:Rubinstein}.
 $F'$ is in standard form,
and let $D_1, D_2$ be as in Definition \ref{definition:standard}.
 Set $D_1 = F_0$, and let $F_i$ be the surface
obtained from $F_{i-1}$ by the $i$-th band sum operation
in Definition \ref{definition:standard}.
 The tree ${\mathbb D}_2$ contains 
a unique minimal edge-path $\gamma(p,q)$ connecting $0/1$ to $p/q$,
in which every edge appears at most once.
 Let $0/1 = r_0, r_1, \cdots, r_k = p/q$ be the vertices
which $\gamma(p,q)$ passes in this order.
 Then the slope of the boundary circle $\partial F_i$ is $r_i$,
and $F' = F_k \cup D_2$.
 Moreover, the minimum crosscap number is equal to $k$, 
the number of the band sum operations.
 Suppose $q >0$.
 Let $r_i=[a_0, a_1, \cdots, a_n]$ be the standard continued fraction expansion.
 Then $r_{i-1}=[a_0, a_1, \cdots, a_n -2]$ 
for any integer $i$ with $1 \le i \le k$.
\end{theorem}

 Section \ref{section:GeometricalIncompressibility} 
contains the proof of Remark \ref{remark:NonOriIncpr},
and Section \ref{section:StandardForm}
that of the former half of Theorem \ref{theorem:Rubinstein}
for self-containedness.
 We prove 
Theorem \ref{theorem:BandSumD2}
in Section \ref{section:SurgeryAndSlope}.
 In Section \ref{section:D2},
Theorem \ref{theorem:tree} is shown.
 Theorems \ref{theorem:NewFormula} and \ref{theorem:main}
are proved in Section \ref{section:MainTheorem}.

\section{Geometrical incompressibility}\label{section:GeometricalIncompressibility}

\begin{lemma}\label{lemma:MiniNonSepIsIncpr}
 Let $M$ be a compact, connected $3$-manifolds,
and $F$ a non-separating closed surface 
of minimum crosscap number in $M$.
 Then $F$ is geometrically incompressible.
\end{lemma}

\begin{proof}
 We assume, for a contradiction, that $F$ is geometrically compressible. 
 Let $D$ be a compressing disk of $F$.
 We take a tubular neighborhood $N(D) \cong D \times I$
so that $N(D) \cap F = (\partial D) \times I$.
 We perform a surgery on $F$ along $D$ to obtain a new surface $F'$,
that is, we set $F' = (F - (\partial D) \times I) \cup (D \times \partial I)$.

 First, we consider the case 
where $\partial D$ is separating in $F$.
 Then $F'$ consists of two connected components, say, $F_1$ and $F_2$,
and $\chi (F) + 2 = \chi (F_1) + \chi (F_2) \cdots$ (i), 
where $\chi$ denotes the Euler characteristic.
 We assume, for a contradiction, 
that both $F_1$ and $F_2$ are separating in $M$.
 Then $F_1$ separates $M$ into two components $M_{1+}$ and $M_{1-}$, 
and $F_2$ into $M_{2+}$ and $M_{2-}$.
 Without loss of generality,
we can assume that $F_j \subset M_{i+}$ for $\{ i,j \} = \{1,2 \}$.
 Then $N(D) \subset M_{1+} \cap M_{2+}$, 
and $F$ separates $M$ 
into $M_{1-} \cup N(D) \cup M_{2-}$ and $(M_{1+} \cap M_{2+}) - N(D)$.
 This contradicts that $F$ is non-separating in $M$.
 Hence either $F_1$ or $F_2$, say, $F_1$ is non-separating.
 Because $D$ is a compressing disk,
$F_2$ is not a sphere and $\chi(F_2) \le 1$,
 Hence the equation (i) implies $\chi(F) < \chi(F_1)$.
 This contradicts the minimality of $\chi(F)$.

 Next, we consider the case where $\partial D$ is non-separating in $F$.
 Let $F_3$ be the surface resulting from the surgery along $D$.
 We assume, for a contradiction, 
that $F_3$ separates $M$ into two connected components, say,
$M_{3+}$ and $M_{3-}$.
 Without loss of generality, 
we assume $D \subset M_{3-}$ after the surgery. 
 Then $F$ separates $M$ into $M_{3+} \cup N(D)$ and $M_{3-} - N(D)$.
 This contradicts that $F$ is non-separating.
 Thus $F_3$ is non-separating in $M$.
 Since $\chi(F)+2 = \chi(F_3)$,
we obtain $\chi(F)<\chi(F_3)$.
 This contradicts the minimality of $\chi(F)$ again.
\end{proof}

\begin{lemma}\label{lemma:NonOriMiniIsIncpr}
 In a lens space,
a non-orientable closed surface with the minimum crosscap number 
is geometrically incompressible.
\end{lemma}

 Note that a similar thing does not hold for $S^2 \times S^1$.

\begin{proof}
 Since $L(p,q)$ is an orientable $3$-manifold,
and since the $2$-dimensional homology $H_2 (L(p,q); {\Bbb Z}) =0$,
a closed surface $F$ in $L(p,q)$ is non-orientable
if and only if $F$ is non-separating. 
 Hence non-orientable closed surface of the minimum crosscap number
is geometrically incompressible in $L(p,q)$ 
by Lemma \ref{lemma:MiniNonSepIsIncpr}.
\end{proof}

\section{Standard form}\label{section:StandardForm}


 In this section, 
the next propostion due to Rubinstein, 
restated using the terminology $\lq\lq$band sum",
is shown for self-containedness.
 The proof is almost 
the same as the original one in \cite{R3}.

\begin{proposition}\label{proposition:StandardForm}
{\rm (Rubinstein, Proposition 11 in \cite{R3})}
 In $L(p,q) \cong V_1 \cup_{\partial} V_2$, 
any geometrically incompressible connected closed surface 
not homeomorphic to the $2$-sphere
is isotopic to a surface in standard form.
\end{proposition}

\begin{lemma}\label{lemma:FcapV2Incpr}
 Let $F$ be a geometrically incompressible, 
connected closed surface in $L(p,q)$.
 We assume that $F \cap V_1$ is a disjoint union of meridian disks of $V_1$.
 If the number of meridian disks of $F \cap V_1$ is minimum 
up to isotopy of $F$ in $L(p,q)$, 
then the surface $S=F \cap V_2$ is geometrically incompressible in $V_2$.
\end{lemma}

\begin{proof}
 We assume, for a contradiction, that $S$ is geometrically compressible. 
 Let $D$ be a compressing disk of $S$.
 Because $F$ is geometrically incompressible,
$D$ is not a compressing disk of $F \subset L(p,q)$,
and $\partial D$ bounds a disk $D'$ in $F$.
 Note that $D'$ intersects $V_1$
since $D$ is a compressing disk of $S$ in $V_2$.
 As is well-known, a lens space is irreducible,
and hence the sphere $D \cup D'$ bounds a $3$-ball, say, $B$ in $L(p,q)$.
 We move $F$ 
by isotoping $D'$ along $B$ onto $D$.
 The number of meridian disks of $F \cap V_1$ decreases.
 This is a contradiction.
 Hence $S$ is geometrically incompressible.
\end{proof}

 Let $M$ be a $3$-manifold with boundary,
and $S$ a $2$-manifold properly embedded in $M$.
 A disk $D$ in $M$ is called a {\it $\partial$-compressing disk}
if $\alpha = D \cap S$ is a subarc of $\partial D$,
$\beta = D \cap \partial M$ is also a subarc of $\partial D$,
$\partial D = \alpha \cup \beta$, 
$\alpha \cap \beta = \partial \alpha = \partial \beta$
and $\alpha$ does not cobound a subdisk of $S$
with a subarc of $\partial S$.
 $S$ is said to be {\it $\partial$-compressible}
if it has a $\partial$-compressing disk.
 Otherwise, $S$ is {\it $\partial$-incompressible}.
 (Any closed surface is considered to be $\partial$-incompressible.)

 When $S$ is $\partial$-compressible,
we can obtain a new surface $S'$ 
by a {\it $\partial$-compression} as below. 
 We take a tubular neighbourhood $N(D)\cong D \times I$ of $D$ in $M$
so that $N(D) \cap S = \alpha \times I$ 
and $N(D) \cap \partial M = \beta \times I$.
 Then we set $S' = (S - \alpha \times I) \cup (D \times \partial I)$.

 The next two lemmas are well-known. We omit their proofs.

\begin{lemma}\label{lemma:IncprInSolidTorus}, 
 Let $V$ be a solid torus,
and $S$ be a (possibly disconnected) $2$-manifold properly embedded in $V$.
 If $S$ is geometrically incompressible and $\partial$-incompressible in $V$,
then $S$ is a disjoint union of spheres and disks.
\end{lemma}

\begin{lemma}\label{lemma:CompressIncpr}
 Let $M$ be a $3$-manifold with boundary,
$S$ a (possibly disconnected) $2$-manifold properly embedded in $M$,
$S'$ a $2$-manifold obtained from $S$ by a $\partial$-compression.
 If $S$ is geometrically incompressible in $M$, then so is $S'$.
\end{lemma}

\begin{proof}[Proof of Proposition \ref{proposition:StandardForm}]
 Assume that $F$ is isotoped
so that it intersects $V_1$
in the minimum number of meridian disks of $V_1$. 
 If $F \cap V_1 = \emptyset$,
then $F$ is a geometrically incompressible closed surface in $V_2$,
and hence is a sphere, which is a contradiction.
 So $F$ intersects $V_1$ in at least one meridian disk.
 The surface $S = F \cap V_2$ is geometrically incompressible in $V_2$ 
by Lemma \ref{lemma:FcapV2Incpr}.
 We assume, for a contradiction, 
that $S$ is $\partial$-incompressible in $V_2$. 
 Then $S$ is a disjoint union of spheres and disks 
by Lemma \ref{lemma:IncprInSolidTorus}, 
and hence $F$ is a union of spheres.
 This contradicts that $F$ is connected and not a sphere. 
 Hence $S$ is $\partial$-compressible. 

\vspace{2mm}
\noindent
{\it Claim}: 
$F \cap V_1$ is a single meiridian disk in $V_1$. 

\begin{proof}[Proof of Claim]
 We assume, for a contradiction, 
that $F \cap V_1$ consists of two or more meridian disks in $V_1$. 
 Let $Q$ be a $\partial$-compressing disk of $S$. 
 We move $F$ by an isotopy along $Q$, to obtain a surface $F'$. 
 This isotopy causes 
a $\partial$-compressing operation on $S = F \cap V_2$ along $Q$ in $V_2$,
and a band sum operation on $F \cap V_1$ 
along the arc $\beta = Q \cap \partial V_1$ in $V_1$. 
 There are two cases:
the endpoints of $\beta$ are contained
in either distinct two meridian disks of $F \cap V_1$,
or a single meridian disk of $F \cap V_1$.

 In the first case, 
the band joins the two meridian disks.
 They are deformed into a peripheral disk, say, $R$ in $V_1$,
that is, $R$ is isotopic into $\partial V_1$
with $\partial R$ fixed. 
 Hence we can move $F$ near $R$ 
so that $R$ is pushed out of $V_1$.
 This decreases the number of meridian disks of $F \cap V_1$,
which is a contradiction. 

 We consider the second case,
where $\partial \beta$ is contained in a meridian disk $D$ of $V_1$.
 The band sum is not essential
since boundary circles of other meridian disks 
prevent $\beta$ from connecting both sides of the circle $\partial D$.
 The endpoints of $\beta$ separates the circle $\partial D$ 
into two subarcs. 
 One of them and $\beta$ together form an inessential circle
which bounds a disk, say, $E$ on $\partial V_1$. 
 $F' \cap E = \partial E$
after the isotopy of $F$ along $Q$.
 Since $S' = F' \cap V_2$ is geometrically incompressible 
by Lemma \ref{lemma:CompressIncpr},
there is a disk $E' \subset S'$ with $\partial E'= \partial E$.
 $E'$ is a connected component of $S'$, 
and $\partial E'$ contains precisely one of two copies of $\beta$.
 Hence the disk $E'$ contains exactly one of two copies of $Q$,
and cl\,$(E' - Q)$ is a disk, 
which contradicts that $Q$ is a $\partial$-compressing disk of $S$. 
 This completes the proof of Claim.
\end{proof}

 Now, $F \cap V_1$ is a single meridian disk of $V_1$. 
 The surface $S = F \cap V_2$ is connected 
since $S$ is obtained from the closed surface $F$ 
by removing the meridian disk $F \cap V_1$.
 $S$ is geometrically incompressible and $\partial$-compressible in $V_2$. 
 Let $E_1$ be a $\partial$-compressing disk of $S$. 
 We move $F$ by an isotopy along $E_1$ to obtain a new surface, say, $F_1$.
 Then $S_1 = F_1 \cap V_2$ is obtained from $S$
by a $\partial$-compression along $E_1$,
and hence is geometrically incompressible 
by Lemma \ref{lemma:CompressIncpr}. 
 $F_1 \cap V_1$ is obtained from $F \cap V_1$ 
by a band sum along the arc $\beta_1 = \partial E_1 \cap \partial V_1$.

 In the case where both ends of the band 
are in the same side of the meridian disk $F \cap V_1$,
we obtain a contradiction 
by a similar argument as that in the latter part of the proof of Claim.
 Hence the ends of the band 
are in distinct sides of the meridian disk $F \cap V_1$. 
 Then the band sum along $\beta_1$ is essential, 
and the boundary of the surface $F_1 \cap V_1$ 
is an essential circle on $\partial V_1$. 
 $F_1 \cap V_1$ and $S_1 = F_1 \cap V_2$ are both connected surfaces 
since $F_1$ is connected and $F_1 \cap \partial V_1$ is a single circle. 
 If $S_1$ is $\partial$-incompressible, 
then $S_1$ is a disjoint union of spheres, meridian disks and peripheral disks 
by Lemma \ref{lemma:IncprInSolidTorus}. 
 Since $\partial (F \cap V_2)$ is an essential circle,
$S_1 = F_1 \cap V_2$ is a meridian disk of $V_2$. 
 Thus we obtained the desired conclusion in this case.

 We consider the case where $S_1$ is $\partial$-compressible in $V_2$
after the isotopy along the disk $E_1$.
 Let $E_2$ be a $\partial$-compressing disk of $S_1$.
 We move $F_1$ along $E_2$ to obtain a new surface $F_2$.
 The surface $S_2 = F_2 \cap V_2$ is obtained from $S_1$
by a $\partial$-compression along $E_2$,
and $F_2 \cap V_1$ is obtained from $F_1 \cap V_1$
by a band sum along the arc $\beta_2 = \partial E_2 \cap \partial V_1$.
 If the band sum along $\beta_2$ is inessential, 
then we obtain a contradiction 
by the same argument 
as that in the latter part of the proof of Claim.
 Hence the band sum along $\beta_2$ is essential,
$F_2 \cap \partial V_2$ is an essential circle in $\partial V_2$,
and $S_2$ and $F_2 \cap V_1$ are both connected surfaces.
 $S_2$ is geometrically incompressible in $V_2$
by Lemma \ref{lemma:CompressIncpr}.
 If $S_2$ is $\partial$-incompressible, 
then $S_2$ is a meridian disk,
and we are done.
 If $S_2$ is $\partial$-compressible, 
then we perform an isotopy along a $\partial$-compressing disk, say, $E_3$. 

 As long as $S_{k-1}$ is $\partial$-compressible 
we repeat an isotopy 
deforming $F_{k-1}$ along a $\partial$-compressing disk $E_k$ 
into a surface $F_k$, 
and set $S_k= F_k \cap V_2$. 
 This repetition terminates in finite number of times
because the Euler characteristic of $S_k$
is larger than that of $S_{k-1}$ by one.
 For some natural number $m \in {\Bbb N}$, 
the surface $S_m$ is $\partial$-incompressible, 
and hence is a meridian disk in $V_2$.
 This completes the proof.
\end{proof}

\section{surgeries on circles and slopes}\label{section:SurgeryAndSlope}

 We show in this section 
that a band sum operation in Definition \ref{definition:standard}
corresponds to an edge of ${\mathbb D}_2$.
 From Definition \ref{definition:SurgeryOnCircle} 
through Definition \ref{definition:dual},
$H$ denotes a surface,
$C$ an embedded circle in $H$,
and $\beta$ an embedded arc in $H$
such that $\beta \cap C = \partial \beta$.

\begin{definition}\label{definition:SurgeryOnCircle}
 A {\it surgery} on $C$ along $\beta$ is an operation
which deforms $C$ to a new circle(s) $C'$ as below.
 We take 
a tubular neighbourhood $N(\beta) \cong \beta \times I$ of $\beta$ in $H$
so that $N(\beta) \cap C = (\partial \beta) \times I$.
 Then $C' = (C - (\partial \beta) \times I) \cup (\beta \times \partial I)$
is either a cirle or a disjoint union of two circles.
\end{definition}

\begin{remark}\label{remark:SurgeryBandSum}
 Let $F$ be a surface properly embedded in a $3$-manifold $M$.
 Suppose $\partial F$ is a circle.
 Then, setting $H = \partial M$ and $C = \partial F$,
a surgery operation on $C$ along $\beta$ 
coincides with the deformation of $\partial F$
under the band sum operation on $F$ along $\beta$.
\end{remark}

 The next lemma and Corollary \ref{corollary:SurgeryOnCircleSlope1}
were pointed out and used 
in the proof of Theorem 13 in \cite{R3} by Rubinstein.

\begin{lemma}\label{lemma:SurgeryOnCircleSlope1}
 Suppose that $H$ and $C$ are oriented.
 Assume that $\beta$ connects both sides of $C$.
 By performing a surgery on $C$ along $\beta$, 
we obtain a single circle, say, $C'$. 
 For an arbitrary orientation of $C'$, 
the algebraic intersection number $C' \cdot C$ is equal to $\pm 2$,
and hence $C'$ is essential in $H$. 
\end{lemma}

\begin{figure}
\includegraphics[width=8cm]{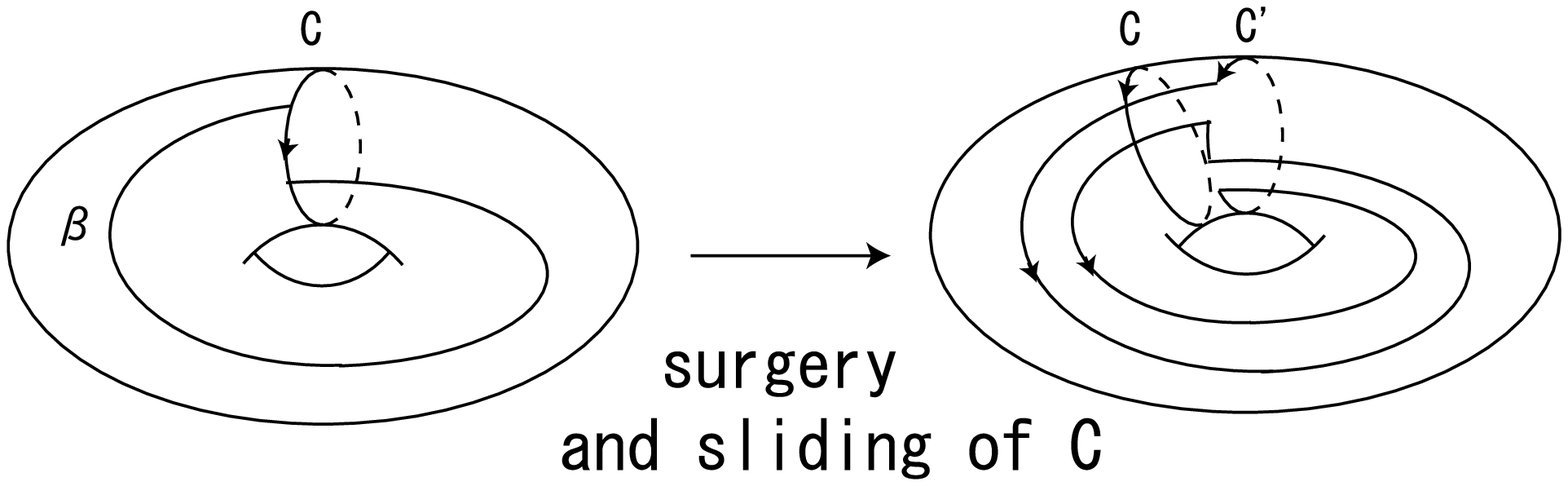}
\caption{}
\label{fig:CSurgeryIntersection2}
\end{figure}

\begin{proof}
 In Definition \ref{definition:SurgeryOnCircle} 
of surgery on $C$ along $\beta$, 
$C - N(\beta)$ consists of two arcs, say, $C_1$ and $C_2$.
 Since $\beta$ connects both sides of $C$ and since $H$ is orientable,
each of the two arcs of $\beta \times \partial I$ 
connects an endpoint of $C_1$ and that of $C_2$. 
 Hence the surgery yields a single circle, say, $C'$. 
 For an arbitrary orientation of $C'$, 
the induced orientation 
of the two subarcs of $\beta \times \partial I$ of $C'$
are parallel
as a pair of opposite sides of the quadrilateral $\beta \times I$.
 After the surgery, we slide the original circle $C$ a little
so that $C$ intersects $\beta \times I$ 
in an arc of the form $p \times I$, where $p$ is a point of $\beta$.
 See Figure \ref{fig:CSurgeryIntersection2}. 
 Then $C$ intersects $C'$ at the endpoints of $p \times I$,
and the signs of them coincide.
 Hence we have $C \cdot C' = \pm 2$, and $C'$ is essential in $H$. 
\end{proof}

\begin{definition}\label{definition:EssentialSurgeryOnCircleSlope}
 The two points $\partial \beta$ 
separate $C$ into two arcs, say, $C_1, C_2$. 
 If none of the circles $\beta \cup C_1$ and $\beta \cup C_2$ 
bounds a disk in $H$, 
then we say that the surgery on $C$ along the arc $\beta$ is {\it essential}. 
\end{definition}

\begin{corollary}\label{corollary:SurgeryOnCircleSlope1}
 If $H$ is a torus,
and if a surgery on $C$ along the arc $\beta$ is essential, 
then it yields a circle $C'$
such that $C \cdot C' = \pm 2$.
 Hence $C'$ is essential in $H$.
\end{corollary}

\begin{proof}
 The arc $\beta$ connects both sides of $C$
since $H$ is a torus and the surgery is essential. 
 Hence we obtain the desired conclusion 
by Lemma \ref{lemma:SurgeryOnCircleSlope1}. 
\end{proof}

\begin{definition}\label{definition:dual}
 We assume that 
a single circle $C'$ is obtained 
by a surgery on $C$ along the arc $\beta$ in $H$.
 Then we can recover $C$ by performing a {\it dual surgery} on $C'$ 
as below. 
 Let $p$ be a point in the interior of the arc $\beta$. 
 $C$ can be restored 
by a surgery on $C'$ along the arc $(p \times I) \subset (\beta \times I)$,
where $\beta \times I$ is a square 
as in Definition \ref{definition:SurgeryOnCircle}. 
\end{definition}

\begin{theorem}\label{thorem:SurgeryOnCircleSlope2}
 Let $T$ be a torus, 
and $C$ and $C'$ oriented circles in $T$. 
 Then (1) and (2) below are equivalent.
\begin{enumerate}
\item[(1)] 
We can obtain $C'$ by performing a surgery on $C$ along an arc.
\item[(2)] 
$C' \cdot C = \pm 2$.
\end{enumerate}
\end{theorem}

\begin{proof}
 First, we assume (1) to show (2). 
 $C'$ can be obatined by a surgery on $C$ along an arc $\beta$.
 If the surgery is inessential, 
then it yields two circles one of which bounds a disk in $T$.
 This contradicts that $C'$ is a single circle.
 Hence the surgery is essential,
and (2) holds by Corollary \ref{corollary:SurgeryOnCircleSlope1}. 

\begin{figure}
\includegraphics[width=8cm]{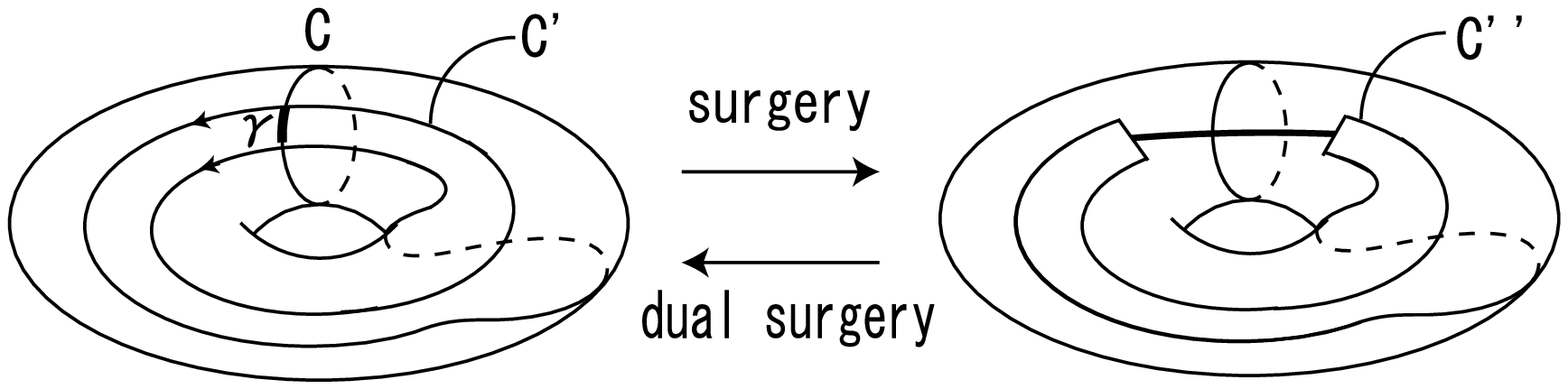}
\caption{}
\label{fig:DualSurgery}
\end{figure}

 Next, we prove (1) under the assumption that (2) holds. 
 It is well-known that 
we can place $C$ and $C'$
so that $C$ and $C'$ intersect at precisely two points of the same sign
since $C \cdot C' = \pm 2$ and $T$ is a torus. 
 The two intersection points separate $C$ into two arcs. 
 Let $\gamma$ be one of them.
 Since the signs of the two intersection points coincide,
a surgery on $C'$ along the arc $\gamma$
yields a single circle, say, $C''$.
 See Figure \ref{fig:DualSurgery}. 
 Because we have already proven (1) implies (2), 
$C' \cdot C'' = \pm 2$, and $C''$ is essential. 
 Then $C''$ is parallel to $C$ in the torus $T$ 
since $C''$ is essential and does not intersect $C$. 
 $C'$ is obtained by a dual surgery on $C''$, and hence on $C$. 
\end{proof}

 Theorem \ref{thorem:SurgeryOnCircleSlope2} 
together with Lemma \ref{lemma:EvenInherit} in the next section
shows Theorem \ref{theorem:BandSumD2}.

\section{${\Bbb D}_2$}\label{section:D2}

 In this section, we prove the graph ${\mathbb D}_2$ is a tree.

\begin{lemma}\label{lemma:EvenInherit}
 Let $p/q$ and $r/s$ be irreducible fractional numbers.
 If $p$ is even and $d(p/q, r/s) = 2$,
then $r$ is also even.
\end{lemma}

\begin{proof}
 Since $\pm 2 = \pm d (p/q, r/s) = ps-qr$,
and since $p$ is even, 
$qr = ps \mp 2$ is also even.
 Hence either $q$ or $r$ is even.
 Because $p/q$ is irreducible, $r$ is even.
\end{proof}

\begin{lemma}
 Let $p/q$ and $r/s$ be irreducible fractional numbers
with $d(p/q, r/s) = 2$.
 If $p/q > 0$ and $r/s < 0$,
then $p/q = 1$ and $r/s=-1$.
\end{lemma}

\begin{proof}
 We can assume $p,q,s >0$ and $r < 0$.
 Then $ps \ge 1$ and $s(-r) \ge 1$,
and $ps + s(-r) \ge 2$.
 On the other hand, $ps + s(-r) = \pm d(p/q, r/s) = \pm 2$.
 Hence we have $ps =1$ and $s(-r)=1$.
 Since $p,q,r,s$ are integers, the lemma follows.
\end{proof}

 Because neither $1/1$ nor $-1/1$ is a vertex of ${\mathbb D}_2$,
the next corollary holds.

\begin{corollary}\label{corollary:+side-side}
 The graph ${\mathbb D}_2$ has no edge
connecting the left hand side and the right hand side of the Poincar\'{e} disk.
\end{corollary}

\begin{notation}
 Let ${\mathbb D}_{2+}$ denote the right half of the graph ${\mathbb D}_2$.
 That is, 
vertices of ${\mathbb D}_{2+}$ are 
those of ${\mathbb D}_2$ corresponding to non-negative rational numbers
and edges of it are those of ${\mathbb D}_2$ 
connecting pairs of vertices of ${\mathbb D}_{2+}$.
 Similarly, we define ${\mathbb D}_{2-}$
as the left half of ${\mathbb D}_2$.
 Then ${\mathbb D}_2 = {\mathbb D}_{2-} \cup {\mathbb D}_{2+}$,
and ${\mathbb D}_{2-} \cap {\mathbb D}_{2+} = \{ 0/1 \}$.
\end{notation}

\begin{lemma}\label{lemma:symmetry}
 ${\mathbb D}_2$ is symmetric about the line connecting $0/1$ and $1/0$.
 Precisely, 
$d(p/q, p'/q') = d(-p/q,-p'/q')$
for any pair of 
irreducible fractional numbers $p/q$ and $p'/q'$. 
\end{lemma}

\begin{proof}
$d(-\displaystyle\frac{p}{q}, -\displaystyle\frac{p'}{q'}) 
= |{\rm det} \left( \begin{array}{cc} -p & -p' \\ q & q' \end{array} \right) |
= |-{\rm det} \left( \begin{array}{cc} p & p' \\ q & q' \end{array} \right) |
= |{\rm det} \left( \begin{array}{cc} p & p' \\ q & q' \end{array} \right) |
= d(\displaystyle\frac{p}{q}, \displaystyle\frac{p'}{q'}) $
\end{proof}

\begin{notation}
 For a finite sequence of real numbers 
${\bf a} = (a_1, a_2, \cdots, a_n)$, 
let $\phi_{\bf a} : 
{\mathbb R} \cup \{ \infty \} \rightarrow {\mathbb R} \cup \{ \infty \}$ 
denote a homeomorphism defined by 
$\phi_{\bf a}(x) = [a_n, a_{n-1}, \cdots, a_2, a_1, x]$.
 If $n=0$,
then $\phi_{\mathcal A}(x) = x$.
\end{notation}

\begin{remark}\label{remark:order}
 We set ${\mathbb R}_+ = \{ x \in {\mathbb R} \ |\ x \ge 0 \}$.
 For any $r \in {\mathbb R}_+$,
the map $x \mapsto x+r$ preserves the order 
in ${\mathbb R}_+ \cup \{ \infty \}$
(i.e., $x+r > y+r$ if $x, y \in {\mathbb R}_+ \cup \{ \infty \}$ and $x>y$),
while the map $x \mapsto 1/x$ reverses it.
 Hence,
for a finite sequence of non-negative real numbers 
${\bf a} = (a_1, a_2, \cdots, a_n)$, 
$\phi_{\bf a}$ preserves the order in ${\mathbb R}_+ \cup \{ \infty \}$ 
when $n$ is even,
and reverses it when $n$ is odd.
\end{remark}

\begin{lemma}\label{lemma:BunsuuKansuu}
 For any finite sequence of integers
${\bf a} = (a_1, a_2, \cdots, a_n)$, 
there are integers
$\alpha, \beta, \gamma$ and $\delta$
such that for any irreducible fractional number $p/q$, 
(1) 
$\phi_{\bf a}(p/q) = 
\dfrac{\alpha p + \beta q}{\gamma p + \delta q}$, 
(2) 
$\alpha p + \beta q$ and $\gamma p + \delta q$ are coprime
and (3) 
det$\left( \begin{array}{cc} 
\alpha & \beta \\ \gamma & \delta 
\end{array} \right) = (-1)^n$.
 Moreover, if $a_1, a_2, \cdots, a_n$ are non-negative,
we can take $\alpha, \beta, \gamma, \delta$ to be non-negative.
\end{lemma}

\begin{proof}
 We show this lemma by induction on $n$.
 When $n=0$, 
$\phi_{\mathcal A} (p/q) = p/q 
= \dfrac{1 \cdot p + 0 \cdot q}{0 \cdot p + 1 \cdot q}$,
and the lemma holds.

 We assume that
there are integers $\alpha', \beta', \gamma', \delta'$ as above
for $(a_1, a_2, \cdots, a_{k-1})$.

 When $n = k$,
$\phi_{\mathcal A}(p/q)
= [a_k, a_{k-1}, \cdots, a_2, a_1, p/q]$
\newline
$= [a_k, \displaystyle\frac{\alpha' p + \beta' q}{\gamma' p + \delta' q}]
= a_k + 
\displaystyle\frac{\gamma' p + \delta' q}{\alpha' p + \beta' q}
= 
\dfrac{(a_k \alpha' + \gamma')p+(a_k \beta' + \delta')q}{\alpha' p + \beta' q}$.
 We set 
$\alpha = a_k \alpha' + \gamma'$, 
$\beta = a_k \beta' + \delta'$, 
$\gamma = \alpha'$
and $\delta = \beta'$.

 Then, 
$\alpha p + \beta q$ and $\gamma p + \delta q$ are coprime 
since
$\alpha p + \beta q = a_k (\gamma p + \delta q) + (\gamma' p + \delta' q)$
and since 
$\gamma p + \delta q = \alpha' p + \beta' q$ and $\gamma' p + \delta' q$
are coprime by the assumption of induction.

 In addition,
${\rm det}
\left(
\begin{array}{cc}
\alpha & \beta \\
\gamma & \delta \\
\end{array}
\right) 
=
{\rm det}
\left(
\begin{array}{cc}
a_k \alpha' + \gamma' & a_k \beta' + \delta' \\
\alpha' & \beta' \\
\end{array}
\right) 
= 
{\rm det}
\left(
\begin{array}{cc}
\gamma' & \delta' \\
\alpha' & \beta' \\
\end{array}
\right) 
=$
\newline
$- 
{\rm det}
\left(
\begin{array}{cc}
\alpha' & \beta' \\
\gamma' & \delta' \\
\end{array}
\right) 
= -(-1)^{k-1}$,
where we obtain the first equation 
by subtracting $a_k$ times the second row from the first row,
and the last equation by the assumption of induction.

 If $a_1, a_2, \cdots, a_k$ are non-negative,
then $\alpha = a_k \alpha' + \gamma'$, 
$\beta = a_k \beta' + \delta'$, 
$\gamma = \alpha'$
and $\delta = \beta'$
are non-negative
since $a_k$ is non-negative
and $\alpha', \beta', \gamma'$ and $\delta'$ are non-negative
by the assumption of induction.
\end{proof}

\begin{lemma}\label{lemma:distance}
 For any finite sequence of integers
${\bf a} = (a_1, a_2, \cdots, a_n)$, 
the map $\phi_{\bf a}$ preserves the distance.
\end{lemma}

\begin{proof}
 Let $p/q$ and $r/s$ be irreducible fractional numbers.
 As in Lemma \ref{lemma:BunsuuKansuu},
there are integers $\alpha, \beta, \gamma$ and $\delta$
with $\phi_{\bf a} (t/u) = (\alpha t + \beta u)/(\gamma t + \delta u)$
for any irreducible fractional number $t/u$.
 Hence
$d(\phi_{\bf a}(p/q), \phi_{\bf a}(r/s))
=|{\rm det} \left( \begin{array}{cc} 
\alpha p + \beta q & \alpha r + \beta s \\
\gamma p + \delta q & \gamma r + \delta s \\
\end{array} \right)|
=|{\rm det} \left( \begin{array}{cc} 
\alpha & \beta \\
\gamma & \delta \\
\end{array} \right)
\left( \begin{array}{cc} 
p & r \\
q & s \\
\end{array} \right)|
=|(-1)^n {\rm det}
\left( \begin{array}{cc} 
p & r \\
q & s \\
\end{array} \right)|
=d(p/q,r/s)$
\end{proof}

\begin{definition}\label{definition:mother}
 For any positive irreducible fractional number $p/q$
with $p$ even and $p, q$ positive,
we define the {\it mother} $M(p/q)$ of $p/q$ as below.
 Let $p/q = [a_0, a_1, \cdots, a_{n-1}, a_n]$
be the standard continued fraction expansion.
 Then we set
$M(p/q) = [a_0, a_1, \cdots, a_{n-1}, a_n - 2]$. 
 Then $M(p/q)$ is a non-negative rational number
which is expressed by an irreducible fractional number
with its numerator even,
which is shown in Lemma \ref{lemma:mother}.
\end{definition}

\begin{remark}
 $d(p/q, M(p/q))=2$ 
by Lemma \ref{lemma:distance} and $d(a_n, a_n-2) = 2$.
 Hence the vertices $p/q$ and $M(p/q)$ are connected by an edge
in the graph ${\mathbb D}_2$.
\end{remark}

\begin{definition}
 For an irreducible fractional number $p/q$ with $p\ge 0$ and $q>0$,
we define {\it size} of $p/q$ as ${\rm size}(p/q) = p + q$.
 For $\infty = 1/0$,
${\rm size}(\infty) = 1 + 0 = 1$.
\end{definition}

\begin{lemma}\label{lemma:mother}
 For any positive irreducible fractional number $p/q$
with $p$ even and $p, q$ positive,
the mother $M(p/q)$ is a non-negative rational number
which is expressed by an irreducible fractional number
with its numerator even.
 Moreover, ${\rm size}(M(p/q)) < {\rm size}(p/q)$
\end{lemma}

\begin{proof}
 As in Definition \ref{definition:mother},
there is a unique continued fraction expansion
\newline
$p/q = [a_0, a_1, \cdots, a_{n-1}, a_n]$.
 For the sequence ${\bf a}=(a_{n-1}, a_{n-2}, \cdots, a_1, a_0)$,
there are non-negative integers $\alpha, \beta, \gamma, \delta$
such that 
${\rm det} \left( \begin{array}{cc}
\alpha & \beta \\
\gamma & \delta \\
\end{array} \right) 
= (-1)^n$
and 
$\phi_{\bf a}(r/s) = 
\dfrac{\alpha r + \beta s}{\gamma r + \delta s}$
for any irreducible fractional number $r/s$
as in Lemma \ref{lemma:BunsuuKansuu}. 

 Then 
$\displaystyle\frac{p}{q} = \phi_{\bf a} (a_n) = 
\displaystyle\frac{\alpha \cdot a_n + \beta \cdot 1}
{\gamma \cdot a_n + \delta \cdot 1}$,
and 
$M(\displaystyle\frac{p}{q}) = \phi_{\bf a} (a_n -2) = 
\dfrac{\alpha (a_n -2) + \beta \cdot 1}
{\gamma (a_n -2) + \delta \cdot 1}$.
 These expression of fractional numbers are irreducible
by Lemma \ref{lemma:BunsuuKansuu}.
 Since $p = \alpha a_n + \beta \cdot 1$ is even,
(the numerator of $M(p/q)$) $= \alpha (a_n -2) + \beta \cdot 1$ is also even.

 Moreover,
size $(p/q) = a_n (\alpha + \gamma) + \beta + \delta$, and
size $(M(p/q)) = (a_n-2)(\alpha + \gamma) + \beta + \delta$. 
 Hence
size $(p/q) - {\rm size}(M(p/q))= 2(\alpha + \gamma) > 0$. 
 Note that either $\alpha > 0$ or $\gamma > 0$ holds
because they are non-negative and 
${\rm det}
\left(
\begin{array}{cc}
\alpha & \beta \\
\gamma & \delta \\
\end{array}
\right) 
\ne 0$.
\end{proof}

\begin{lemma}\label{lemma:connected}
 The graph ${\mathbb D}_2$ is connected.
 In fact, for any irreducible fractional number $p/q$
with $p$ even, $p \ge 0$ and $q > 0$,
the sequence
\newline
$p/q, M(p/q), M(M(p/q)), M(M(M(p/q))), \cdots, M^k(p/q), \cdots $
\newline
reaches $0/1$.
 That is, $M^m (p/q) = 0/1$ for some non-negative integer $m$.
\end{lemma}

\begin{proof}
 By Lemma \ref{lemma:symmetry}
and ${\mathbb D}_{2-} \cap {\mathbb D}_{2+} = \{ 0/1 \}$,
it is enough to show that ${\mathbb D}_{2+}$ is connected.

 If $M^k(p/q) > 0$, 
then we take $M^{k+1}(p/q)$,
which is of smaller size than $M^k(p/q)$ by Lemma \ref{lemma:mother}.
 This repetition terminates at most ${\rm size}(p/q) = p+q$ times.
 Hence ${\mathbb D}_{2+}$ is connected.
\end{proof}

\begin{definition}
 Let $p$, $q$ be coprime integers
with $p$ even, $p \ge 0$ and $q>0$.
 Then an irreducible fractional number $r/s$
is a {\it child} of $p/q$
if $d(r/s, p/q) = 2$
and $r/s$ is not the mother of $p/q$.
\end{definition}

\begin{definition}\label{definition:generation}
 Let $p$ and $q$ be coprime integers with $p$ even, $p \ge 0$ and $q > 0$.
 We say that the irreducible fractional number $p/q$
is of the $k$th {\it generation}
if $M^k(p/q) = 0/1$.
 We set ${\mathbb D}_{2k+}$ 
to be a subgraph of ${\mathbb D}_{2+}$
such that
its vertices are the vertices of ${\mathbb D}_{2+}$
of $k$ or smaller generation
and its edges are those of ${\mathbb D}_{2+}$
with endpoints at vertices of $k$ or smaller generations.
\end{definition}

\begin{definition}
 Let $p$, $q$ be coprime integers
with $p$ even, $p \ge 0$ and $q>0$.
 For the irreducible fractional number $p/q$,
we define the territory $T(p/q)$ of $p/q$ 
as a certain open interval in ${\mathbb R}$ below.
 Let $p/q = [a_0, a_1, \cdots, a_{n-1}, a_n]$
be the standard continued fraction expansion.

 When $n$ is even, 
$T(p/q) = 
([a_0, a_1, \cdots, a_{n-1}, a_n-1], \ 
[a_0, a_1, \cdots, a_{n-1}, \infty])$. 

 When $n$ is odd, 
$T(p/q) = 
([a_0, a_1, \cdots, a_{n-1}, \infty], \ 
[a_0, a_1, \cdots, a_{n-1}, a_n-1])$. 
\end{definition}

\begin{lemma}\label{lemma:child}
 Let $p$, $q$ be coprime positive integers with $p$ even, 
and $p/q = [a_0, a_1, \cdots, a_{n-1}, a_n]$
the standard continued fraction expansion.
 A rational number $r/s$ is a child of $p/q$
if and only if $r/s=[a_0, a_1, \cdots, a_{n-1}, a_n + 2/t]$
for some odd integer $t$ other than $-1$.
 $p/q$ is the mother of its child,
and hence if $p/q$ is of the $g$th generation,
then its child is of the $(g+1)$st generation.
 The territory $T(p/q)$ contains $p/q$ and all the children of $p/q$,
and $M(p/q) \notin T(p/q)$.
\end{lemma}

\begin{proof}
 Let $u/w$ be an irreducible fractional number 
with $d(a_n, u/w)=2, u\ge 0$ and $w>0$. 
 Then $u = a_n w \pm 2$. 
 If $w$ is even, then $u$ is also even, 
which contradicts that $u/w$ is irreducible. 
 Hence $w$ is odd. 
 Dividing both sides by $w$, 
we have $u/w = a_n \pm 2/w$. 

 Hence, by Lemma \ref{lemma:distance},
an irreducible fractional number $r/s$ with $d(r/s, p/q)=2$
is of the form
$r/s=[a_0, a_1, \cdot, a_{n-1}, a_n + 2/t]$
for some odd integer $t$.
 This is the mother of $p/q$ when $t=-1$.

 When $t \ne -1$, 
an easy calculation shows that $p/q$ is the mother of $r/s$.
 For example, when $t \le -5$,
we have a continued fraction expansion
$r/s=[a_0, a_1, \cdots, a_{n-1}, a_n - 1, 1, (-t-3)/2, 2]$.
 Hence the mother of $r/s$ is
$M(r/s)=[a_0, a_1, \cdots, a_{n-1}, a_n - 1, 1, (-t-3)/2, 2-2]
=[a_0, a_1, \cdots, a_n] = p/q$.

 Remark \ref{remark:order} and the sequence of inequalities below imply that 
$p/q \in T(p/q)$,
any child of $p/q$ is in $T(p/q)$
and $M(p/q) \notin T(p/q)$.
\newline
$\infty > a_n + 2/1 > a_n + 2/3 > a_n + 2/5 > \cdots > a_n >$
\newline
\hspace*{2cm}$\cdots > a_n - 2/5 > a_n - 2/3 > a_n -1 > a_n -2 \ge 0$
\end{proof}

\begin{lemma}\label{lemma:territory}
 Let $p$, $q$ be coprime positive integers with $p$ even.
 For any child $r/s$ of $p/q$,
$T(r/s) \subset T(p/q)$ and $p/q \notin T(r,s)$.
 Let $r'/s'$ be another child of $p/q$.
 Then $T(r'/s') \cap T(r/s) = \emptyset$.
\end{lemma}

\begin{proof}
 Let $p/q = [a_0, a_1, \cdots, a_{n-1}, a_n]$
be the standard continued fraction expansion.
 Then $r/s=[a_0, a_1, \cdots, a_{n-1}, a_n + 2/t]$
for some odd integer $t$ other than $-1$ by Lemma \ref{lemma:child}.

 We show the lemma in the case where $n$ is even.
 (If $n$ is odd, then the order in ${\mathbb R}$ is reversed
by the map $\phi_{\bf a}$ 
with ${\bf a} = (a_{n-1}, a_{n-2}, \cdots, a_0)$.
 However, a similar argument as below shows the lemma.)

 An easy calculation shows that
\newline
$T(r/s)= 
([a_0, a_1, \cdots, a_{n-1}, a_n + 2/(t+1)],\ 
 [a_0, a_1, \cdots, a_{n-1}, a_n + 2/(t-1)])$.
\newline
 (Note that $2/(t-1) = \infty$ when $t=1$.)
 For example, when $t \le -5$,
we have a continued fraction expansion
$r/s=[a_0, a_1, \cdots, a_{n-1}, a_n - 1, 1, (-t-3)/2, 2]$.
 Hence 
$T(r/s) = 
([a_0, a_1, \cdots, a_{n-1}, a_n - 1, 1, (-t-3)/2, \infty],\ 
[a_0, a_1, \cdots, a_{n-1}, a_n - 1, 1, (-t-3)/2, 2-1])$,
which coincides with the above open interval.

 Since
$a_n - 2/2 < a_n -2/4 < a_n - 2/6 < \cdots < a_n 
< \cdots < a_n + 2/4 < a_n+ 2/2 < \infty$, 
Remark \ref{remark:order} implies that 
the territory $T(r/s)$ is contained in
$([a_0, a_1, \cdots, a_{n-1}, a_n -1],\ p/q)$ when $t$ is negative,
and 
in $(p/q,\ [a_0, a_1, \cdots, a_{n-1}, \infty])$ when $t$ is positive,
and the lemma holds.
\end{proof}

\begin{proof}[Proof of Theorem \ref{theorem:tree}]
 We show that ${\mathbb D}_2$ is a tree.
 By Lemma \ref{lemma:connected}, ${\mathbb D}_2$ is connected.
 Hence we have only to show that ${\mathbb D}_2$ contains no cycle.
 Since ${\mathbb D}_2 = {\mathbb D}_{2+} \cup {\mathbb D}_{2-}$
and ${\mathbb D}_{2+} \cap {\mathbb D}_{2-} = \{ 0/1 \}$,
and since ${\mathbb D}_2$ is symmetric 
about the line connecting $0/1$ and $1/0$,
it is sufficient to show 
that ${\mathbb D}_{2+}$ contains no cycle.

 Lemmas \ref{lemma:child} and \ref{lemma:territory}
together imply that
${\mathbb D}_{2(m+1)+}$ retracts to ${\mathbb D}_{2m+}$
for any positive integer $m$.
 Hence ${\mathbb D}_{2k+}$ does not contain a cycle.

 ${\mathbb D}_{2+} = \cup_{i=1}^{\infty} {\mathbb D}_{2i+}$
by Lemma \ref{lemma:connected}.
 If ${\mathbb D}_{2+}$ contained a cycle,
then also ${\mathbb D}_{2k+}$ would contain a cycle
for some positive integer $k$.
\end{proof}

\section{The proof of the main theorem}\label{section:MainTheorem}

\begin{lemma}\label{lemma:FEuler}
 Let $F$ be a surface in standard form
as in Definition \ref{definition:standard},
and $b$ the number of the band sum operations.
 Then,
the Euler characteristic of $F$ is calculated by $\chi(F) = 2- b$,
and hence the crosscap number of $F$ is Cr\,$(F) = b$. 
\end{lemma}

\begin{proof}
 We can regard the meridian disk in $V_1$ as a $0$-handle, 
each band as a $1$-handle, 
the meridian disk in $V_2$ as a $2$-handle. 
\end{proof}

\begin{figure}
\includegraphics[width=8cm]{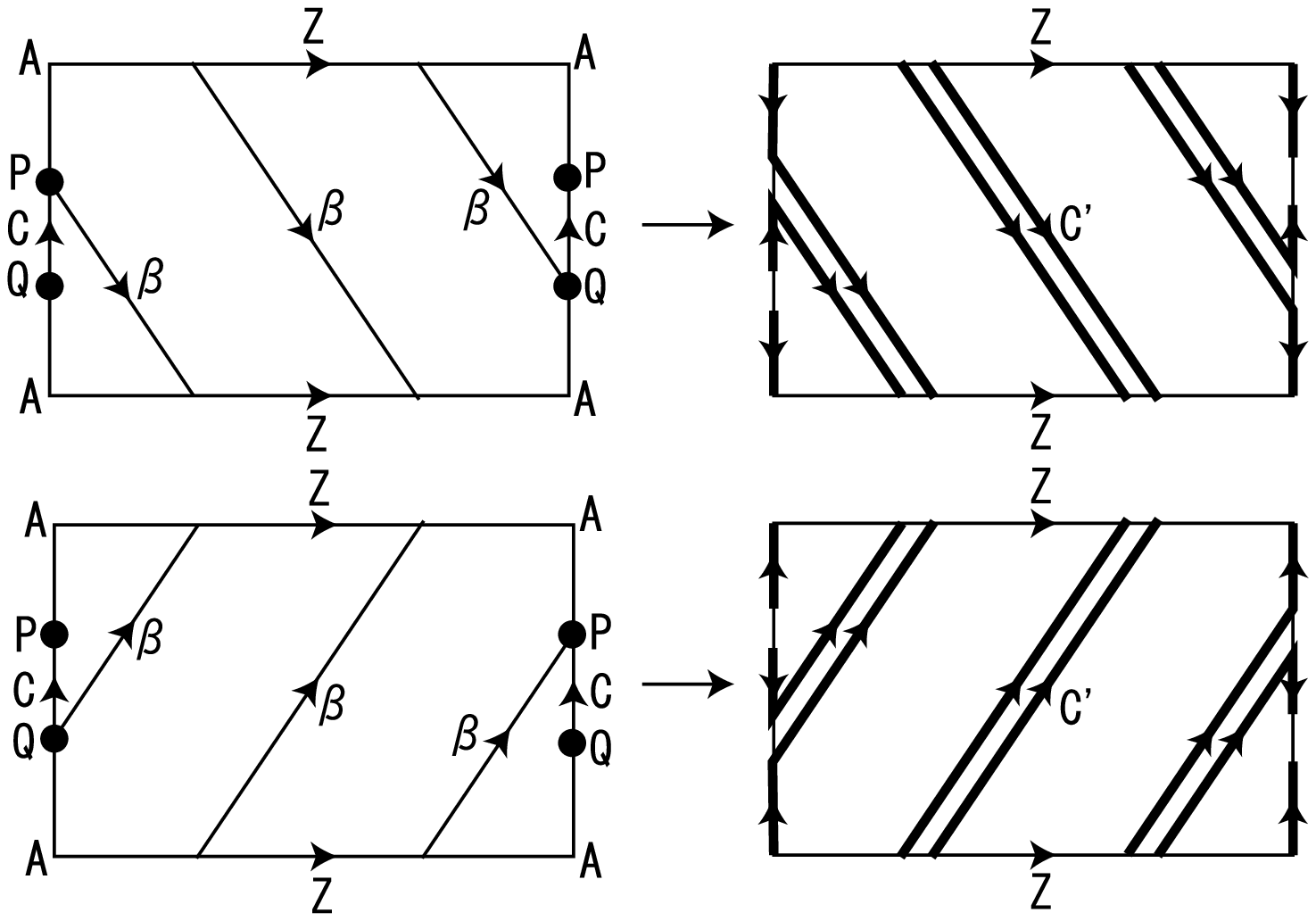}
\caption{}
\label{fig:UniqueBandSum}
\end{figure}

\begin{lemma}\label{lemma:UniqueBandSum}
 Let $V \cong D^2 \times I$ be a solid torus,
$C$ a circle of $(p,q)$-slope on $\partial V$,
and $r/s$ an irreducible fractional number with $d(r/s, p/q)=2$.
 Then there is a unique arc $\beta$
up to ambient isotopy of $V$ fixing $C$ as a set
such that a surgery on $C$ along $\beta$ yields a circle of $(r,s)$-slope.
\end{lemma}

\begin{proof}
 The existence follows from Theorem \ref{thorem:SurgeryOnCircleSlope2}.
 We show the uniqueness.
 We give $C$ an arbitrary orientation.
 There is an oriented circle $Z$ in $\partial V$
such that $Z$ intersects $C$ transversely in a single point, say, $A$,
and $Z \cdot C = + 1$.
 We fix $Z$.
 Let $\beta$ be an arc embedded in $\partial V$
such that $\beta \cap C = \partial \beta$.
 Let $C'$ be the circle obtained from $C$
by a surgery along $\beta$.
 We orient $C'$ so that $C' \cdot C = + 2$,
which induces an orientation of $\beta$.
 It is sufficient to show
that distinct ambient isotopy classes of $\beta$
give distinct intersection numbers $C' \cdot Z$.
 Let $P$ and $Q$ denote endpoints of $\beta$
so that $A, Q, P$ appear in this order on the oriented circle $C$.
 Assume that $\beta$ intersects $Z$ transversely 
in the nimimum number of points
up to ambient isotopy of $V$ fixing $C$ as a set.
 Let $k$ be the minimum number.
 Then $C' \cdot Z = +(2k+1)$ or $-(2k+1)$
according as $\beta$ starts at $P$ or $Q$.
 See Figure \ref{fig:UniqueBandSum},
where the torus $\partial V$ cut along $C \cup Z$ is described. 
\end{proof}

 Note that the argument below shows
the uniqueness of the isotopy class
of geometrically incompressible closed surface 
in $L(p,q)$ with $p$ even.

\begin{proof}[Proof of Theorem \ref{theorem:main}]
 By Proposition \ref{proposition:StandardForm},
$F$ is isotopic to a surface in standard form.
 The transition of slopes by band sum operations
is along an edge-path, say, $\rho$ in ${\Bbb D}_2$
as mentioned in right after Theorem \ref{theorem:BandSumD2}.
 $\rho$ starts at $0/1$ and ends at $p/q$. 

 Since ${\Bbb D}_2$ is a tree by Theorem \ref{theorem:tree}, 
there is a unique minimal edge-path $\gamma (p,q)$
connecting $0/1$ and $p/q$
such that $\gamma (p,q)$ passes each edge of ${\mathbb D}_2$ at most once.
 The union of the edges of $\rho$ contains $\gamma (p,q)$.

 Suppose, for a contradiction, 
that $\rho \ne \gamma (p,q)$.
 Then $\rho$ passes the same edge of ${\mathbb D}_2$ twice consecutively.
 Hence corresponding two band sum operations 
cause mutually dual surgeries on boundary circles on $\partial V_1$
by Lemma \ref{lemma:UniqueBandSum}.
 These two bands together form an annulus
whose core circle bounds a compressing disk $Q$ of the surface $F_i$ in $V_1$
such that $\partial Q$ is non-separating in $F_i$.
 This contradicts that $F$ is geometrically incompressible.

 Thus $\rho = \gamma (p,q)$.
 (This and Lemma \ref{lemma:UniqueBandSum} together
imply that a surface in standard form is unique up to isotopy.)
 Then the theorem follows 
from Lemmas \ref{lemma:FEuler} and \ref{lemma:connected}.
\end{proof}

\begin{proof}[Proof of Theorem \ref{theorem:NewFormula}]
 The number of edges 
in the edge-path $\gamma (p,q)$ in Theorem \ref{theorem:main}
is equal to the number of the band sum operations
and to the minimum crosscap number.

 Let $p/q=[\alpha_n, \alpha_{n-1}, \cdots, \alpha_1, \alpha_0]$ 
be the standard continued fraction expansion.
 If $\alpha'_0 = \alpha_0 = 2\beta_0 + 1$ for some $\beta_0 \in {\mathbb N}$,
then 
$M^{\beta_0} (p/q) =
[\alpha_n, \alpha_{n-1}, \cdots, \alpha_1, 1]
=[\alpha_n, \alpha_{n-1}, \cdots, \alpha_1 +1 ]$.
 We set $\alpha'_1 = \alpha_1 + 1$, which is the last term.
 If $\alpha'_0 = \alpha_0 = 2\beta_0$ for some $\beta_0 \in {\mathbb N}$,
then
$M^{\beta_0} (p/q) =
[\alpha_n, \alpha_{n-1}, \cdots, \alpha_2, \alpha_1, 0]
=[\alpha_n, \alpha_{n-1}, \cdots, \alpha_2, \alpha_1 + 1/0 ]
=[\alpha_n, \alpha_{n-1}, \cdots, \alpha_2, \infty ]$.
 We set $\alpha'_1 = \infty$.

 When $\alpha'_1 = \infty$,
$[\alpha_n, \alpha_{n-1}, \cdots, \alpha_2, \infty]
=[\alpha_n, \alpha_{n-1}, \cdots, \alpha_2 + 1/\infty]
=[\alpha_n, \alpha_{n-1}, \cdots, \alpha_2]$.
 We set $\beta_1 = 0$ and $\alpha'_2 = \alpha_2$.
 When $\alpha'_1 = 2\beta_1 + 1$ $(\beta_1 \in {\mathbb N})$,
we set $\alpha'_2 = \alpha_2 + 1$.
 When $\alpha'_1 = 2\beta_1$ $(\beta_1 \in {\mathbb N})$,
we set $\alpha'_2 = \infty$.

 We repeat operations of going up to the mother as above.
 Let $[\alpha_n, \alpha_{n-1}, \cdots, \alpha_i, \alpha'_{i-1}]$
be the continued fraction expansion of length $n-(i-2)$
which we first reach.
 Then we set 
$\alpha'_i = 
\alpha_i$ (when $\alpha'_{i-1}=\infty$),
$\ \alpha_i+1$ 
(when $\alpha'_{i-1}= 2 \beta_{i-1} +1$ 
for some $\beta_{i-1} \in {\mathbb N}\cup \{ 0 \}$) and
$\ \infty$ 
(when $\alpha'_{i-1}= 2 \beta_{i-1}$ 
for some $\beta_{i-1} \in {\mathbb N}$)
in a similar way as above.

 When we first reach the continued fraction of length one, say, $[\alpha'_n]$,
the non-negative integer $\alpha'_n$ is even
because it is a vertex of ${\mathbb D}_2$.
 Hence we set $\alpha'_n = 2\beta_n$.
 Then $M^{\beta_n} (\alpha'_n) = 0/1$.
 Thus the number of operations of going up to the mother from $p/q$ to $0/1$
is $\sum_{i=0}^n \beta_i$,
which is equal to the crosscap number.
\end{proof}

\section*{Acknowledgement}

The author thanks Koya Shimokawa 
for a helpful comment on the proof of Theorem \ref{theorem:tree},
and her advisor, Chuichiro Hayashi, 
for his encouragement and useful discussions.

\bibliographystyle{amsplain}

\medskip

\noindent
Miwa Iwakura:
Department of Mathematical and Physical Sciences,
Faculty of Science, Japan Women's University,
2-8-1 Mejirodai, Bunkyo-ku, Tokyo, 112-8681, Japan.
miwamathematics@yahoo.co.jp

\end{document}